\documentclass[12pt]{amsart}
\usepackage{}
\usepackage{amscd,amsmath,amsthm,amssymb}
\usepackage{pstcol,pst-plot,pst-3d}%\usepackage[T1]{fontenc}
\usepackage{color}
\usepackage{pstricks}
\usepackage{lineno,xcolor}
\usepackage{tikz}
\DeclareSymbolFont{largesymbol}{OMX}{yhex}{m}{n}
\DeclareMathAccent{\Widehat}{\mathord}{largesymbol}{"62}

\newpsstyle{fatline}{linewidth=1.5pt}
\newpsstyle{fyp}{fillstyle=solid,fillcolor=verylight}
\definecolor{verylight}{gray}{0.97}
\definecolor{light}{gray}{0.9}
\definecolor{medium}{gray}{0.85}
\definecolor{dark}{gray}{0.6}

 \def\G{{\mathcal G}}

 \def\supp{{\mathcal supp}}

 \def\opn#1#2{\def#1{\operatorname{#2}}} % to make operators
 \opn\chara{char} \opn\length{\ell} \opn\pd{pd} \opn\rk{rk}
 \opn\projdim{proj\,dim} \opn\injdim{inj\,dim} \opn\rank{rank}
 \opn\depth{depth} \opn\grade{grade} \opn\height{height}
 \opn\embdim{emb\,dim} \opn\codim{codim}
 
 \opn\Tr{Tr} \opn\bigrank{big\,rank}
 \opn\superheight{superheight}\opn\lcm{lcm}
 \opn\trdeg{tr\,deg}%\emph{
 \opn\reg{reg} \opn\lreg{lreg} \opn\ini{in} \opn\lpd{lpd}
 \opn\size{size} \opn\sdepth{sdepth}
 \opn\link{link}\opn\fdepth{fdepth}\opn\lex{lex}
 \opn\tr{tr}
 \opn\type{type}
 \opn\Borel{Borel}
\opn\cdeg{cdeg}

 \opn\div{div} \opn\Div{Div} \opn\cl{cl} \opn\Cl{Cl}
 \opn\Spec{Spec} \opn\Supp{Supp} \opn\supp{supp} \opn\Sing{Sing}
 \opn\Ass{Ass} \opn\Min{Min}\opn\Mon{Mon}
 \opn\Ann{Ann} \opn\Rad{Rad} \opn\Soc{Soc}
 \opn\Im{Im} \opn\Ker{Ker} \opn\Coker{Coker} \opn\Am{Am}
 \opn\Hom{Hom} \opn\Tor{Tor} \opn\Ext{Ext} \opn\End{End}
 \opn\Aut{Aut} \opn\id{id}
 
 \opn\nat{nat}
 \opn\pff{pf}%   \pf exists already
 \opn\Pf{Pf} \opn\GL{GL} \opn\SL{SL} \opn\mod{mod} \opn\ord{ord}
 \opn\Gin{Gin} \opn\Hilb{Hilb}\opn\sort{sort}
 \opn\PF{PF}\opn\Ap{Ap}
 \opn\aff{aff} \opn
 \con{conv} \opn\relint{relint} \opn\st{st}
 \opn\lk{lk} \opn\cn{cn} \opn\core{core} \opn\vol{vol}  \opn\inp{inp} \opn\nilpot{nilpot}
 \opn\link{link} \opn\star{star}\opn\lex{lex}\opn\set{set}
 \opn\width{wd}
 \opn\Fr{F}
 \opn\QF{QF}
 \opn\G{G}
 \opn\type{type}\opn\res{res}
 \opn\gr{gr}
 
  \def\cdeg{deg}

 \def\pot#1#2{#1[\kern-0.28ex[#2]\kern-0.28ex]}

 \opn\dirlim{\underrightarrow{\lim}}
 \opn\inivlim{\underleftarrow{\lim}}
 \def\Implies{\ifmmode\Longrightarrow \else
         \unskip${}\Longrightarrow{}$\ignorespaces\fi}
 \def\implies{\ifmmode\Rightarrow \else
         \unskip${}\Rightarrow{}$\ignorespaces\fi}
 \def\iff{\ifmmode\Longleftrightarrow \else
         \unskip${}\Longleftrightarrow{}$\ignorespaces\fi}

 \let\:=\colon
 \newtheorem{Theorem}{Theorem}[section]
 \newtheorem{Lemma}[Theorem]{Lemma}
 \newtheorem{Corollary}[Theorem]{Corollary}
 
 \newtheorem{Remark}[Theorem]{Remark}
 
 \newtheorem{Example}[Theorem]{Example}
 
 \newtheorem{Definition}[Theorem]{Definition}

 \let\epsilon\varepsilon
 \let\kappa=\varkappa
 \def\qed{\ifhmode\textqed\fi
       \ifmmode\ifinner\quad\qedsymbol\else\dispqed\fi\fi}
 \def\textqed{\unskip\nobreak\penalty50
        \hskip2em\hbox{}\nobreak\hfil\qedsymbol
        \parfillskip=0pt \finalhyphendemerits=0}
 \def\dispqed{\rlap{\qquad\qedsymbol}}
 
 \opn\dis{dis}
 \def\pnt{{\raise0.5mm\hbox{\large\bf.}}}
 
 \opn\Lex{Lex}

\begin{document}
%\linenumbers
\title {Regularity and projective dimension  of  powers of  edge ideal of the disjoint union of some weighted oriented gap-free bipartite graphs}

\author {Guangjun Zhu$^{^*}$\!\!\!, Li Xu, Hong Wang and Jiaqi Zhang }

\address{Authors¡¯ address:  School of Mathematical Sciences, Soochow
University, Suzhou 215006, P.R. China}
\email{zhuguangjun@suda.edu.cn(Corresponding author:Guangjun Zhu),
\linebreak[4] 1240470845@qq.com(Li Xu), 651634806@qq.com(Hong Wang),\nolinebreak[2] zjq7758258@vip.qq.com(Jiaqi Zhang).}

\dedicatory{ }

\begin{abstract}
In this paper we provide some precise formulas for regularity of powers
of edge ideal of the disjoint union of some weighted oriented gap-free bipartite graphs. For the projective dimension of such an edge ideal, we give its exact formula. Meanwhile, we also give the upper and lower bounds of projective dimension of higher power of such edge ideals. Some examples show that these formulas are related to direction selection.
\end{abstract}

\thanks{* Corresponding author}

%		13H10   	Special types (Cohen-Macaulay, Gorenstein, Buchsbaum, etc.)
%		13D02   	Syzygies, resolutions, complexes
%		05E40   	Combinatorial aspects of commutative algebra
%		16S36   	Ordinary and skew polynomial rings and semigroup rings

%		14M25   	Toric varieties, Newton polyhedra [See also 52B20]
%		13A02   	Graded rings
%		13F20   	Polynomial rings and ideals; rings of integer-valued polynomials
%		13A18   	Valuations and their generalizations
%		06A11   	Algebraic aspects of posets


\maketitle

\setcounter{tocdepth}{1}
%\tableofcontents

\section{Introduction}

\hspace{3mm} An oriented graph $D=(V(D),E(D))$ consists of an underlying simple graph $G$ on which
each edge is given an orientation (i.e., a directed graph without multiple edges nor loops).
 If $\{u,v\}\in E(D)$ is an edge, we write $uv$ for $\{u,v\}$, which is denoted to be the directed edge
where the direction is from $u$ to $v$ and $u$ (resp. $v$) is called the {\em starting}  point (resp. the {\em ending} point).
 A vertex-weighted (or simply, weighted) oriented graph is a triplet $D=(V(D), E(D),w)$, where  $w$ is a weight function $w: V(D)\rightarrow \mathbb{N}^{+}$, where $N^{+}=\{1,2,\ldots\}$.
Some times for short we denote the vertex set $V(D)$ and edge set $E(D)$
by $V$ and $E$ respectively. For any $x_i\in V$, its weight $w(x_i)$ is  denoted by $w_i$ or $w_{x_i}$.

 Let $D=(V,E,w)$ be a weighted  oriented graph with  vertex set $V=\{x_{1},\ldots,x_{n}\}$. We consider the polynomial ring $S=k[x_{1},\dots, x_{n}]$ in $n$ variables over a field $k$. The edge ideal of $D$,  denoted by $I(D)$, is the ideal of $S$ given by
$$I(D)=(x_ix_j^{w_j}\mid  x_ix_j\in E).$$

According to the above definition, the edge ideal $I(D)$ of $D$ is independent of the number of its isolated vertices, we shall always assume that  $D$ has no isolated vertices throughout this paper.
Edge ideals of weighted  oriented graphs arose in the theory of Reed-Muller codes as initial ideals of vanishing ideals
of projective spaces over finite fields \cite{MPV,PS}.
If a vertex $x_i$ of $D$ is a source (i.e., has only arrows leaving $x_i$) we shall always
assume $w_i=1$ because in this case the definition of $I(D)$ does not depend on the
weight of $x_i$.  If  $w_j=1$ for all $j$, then $I(D)$ recovers the usual edge ideal of its (undirected) underlying graph.
Edge ideals of (undirected) graphs have been investigated extensively in the literature \cite{AB,ABS,B1,B2,BBH1,BBH2,BHT,FM,HH3,JNS,MSY,M,RJNP}.
In general, edge ideals of weighted oriented graphs are different from edge ideals of edge-weighted (undirected)
graphs defined by Paulsen and Sather-Wagstaff \cite{PS}.

Let $G$  be a graph without isolated vertices, we say that its two disjoint edges $uv$ and $xy$ form a gap
 if $G$  has no edge with one endpoint in $\{u, v\}$ and the other in $\{x, y\}$.
A graph  without gaps is called gap-free. Equivalently, a graph  $G$ is gap-free if and only if its complement $G^c$
contains no induced $C_4$. Thus, $G$ is gap-free if and only if it does not contain two vertex-disjoint edges as
an induced subgraph. A bipartite graph  is a graph whose vertices can be partitioned into two disjoint $X$, $Y$ such that any edge connects a vertex in $X$ to one in $Y$. Vertex sets $X$, $Y$  are usually called a bipartition of this graph. A complete bipartite graph  is a bipartite graph with  bipartition $X$, $Y$,  every vertex in $X$ is joined to every vertex in $Y$. A star graph
is a complete bipartite graph  with  bipartition $X$, $Y$ such that $|X|=1$ or $|Y |=1$. A  graph  is called gap-free bipartite  graph if it is not only a bipartite graph, but also gap-free.

Every concept that is valid for graphs automatically applies to  oriented graphs too.
For example, the degree of a vertex $x$ in an oriented graph $D$, denoted $d(x)$, is simply the degree of $x$ in its underlying  graph. Likewise, an oriented graph is said to be connected if
its underlying graph is connected; an oriented graph is said to be a gap-free bipartite graph with bipartition $X$, $Y$ if
its underlying graph is a gap-free bipartite graph with bipartition $X$, $Y$.

For a homogeneous ideal  $I\subset S$, its regularity and projective dimension are two central invariants associated to $I$. It is well known that
 $\reg(I^t)$ is asymptotically a linear function for $t\gg  0$, i.e., there exist constants $a$ and $b$ such that for all $t\gg 0$, $\mbox{reg}\,(I^t)=at+b$ (see \cite{CHT,K,TW}). The coefficient $a$ is well-understood, the constants $b$ and $t_0$ are quite mysterious. In this regard, there has been an interest in finding the exact form of the linear
function and determining the stabilization index $t_0$ where $\reg(I^t)$ becomes linear (cf. \cite{AB,ABS,B1,B2,BBH1,BBH2,BHT,JNS,MSY,RJNP}). It turns out that even in the case of monomial ideals it is challenging
to find the linear function and $t_0$ (cf. \cite{Con}).
 In \cite{B3}, Brodmann  showed
that  $\mbox{depth}\,(S/I^t)$ is a constant for $t\gg 0$, and this constant is bounded
above by $n-\ell(I)$, where $\ell(I)$ is the analytic
spread of $I$.  It is shown  in \cite[Theorem 1.2]{HH3} that $\mbox{depth}\,(S/I^t)$ is a nonincreasing
function of $t$ when all powers of $I$ have a linear resolution
and conditions are given in that paper under which all powers of $I$ will
have linear quotients. By Auslander-Buchsbaum formula, we obtain the projective dimension $\mbox{pd}\,(S/I^t)$ is a constant for $t\gg 0$.
In this regard, there has been an interest in determining the smallest value $t_0$ such that $\mbox{pd}\,(S/I^t)$ is a constant for all $t\geq t_ 0$.
 (see  \cite{FM,HH3,JNS,M}).

In this paper, we consider the edge ideal $I(D)$ of a weighted oriented graph $D$. In this case, there exist integers $b$ and $t_0$ such that $\mbox{reg}\,(I^t)=(w+1)t+b$
for all $t\geq t_0$,  where $w=\mbox{max}\,\{w(x)\mid x\in V(D)\}$. We also derive some exact formulas for  projective
dimension and regularity of powers of the edge ideals of some weighted oriented graphs (see, for example, \cite{Z3,Z4,Z5,Z6,Z7}). To the best of our knowledge, there is a few results about  projective
dimension  and regularity of powers of the edge ideal  for a  weighted oriented graph. Our objective in this paper is to express projective dimension  and regularity in terms of combinatorial
invariants of a weighted oriented graph, which is the disjoint union of some   gap-free  bipartite graphs.

 For  a positive integer $\ell$, we set $[\ell]=\{1,\ldots,\ell\}$. Our main results are as follows:
\begin{Theorem}
Let $s$ be a positive integer and $D=\coprod\limits_{i=1}^{s}D_{i}$  the disjoint union of  weighted oriented gap-free bipartite graphs $D_i$ with bipartition $X_i,Y_i$ for all $i\in [s]$.
Let $b_i=\mbox{max}\,\{|X_i|,|Y_i|\}$ and  $N_{D_i}^{+}(x_{i_j})=\{y_{i_1},\ldots,y_{i_{k_j}}\}$ for any $x_{i_j}\in X_i$. If the orientation of $D_i$ is all its edges are directed away from $X_i$,
then
\begin{itemize}
\item[(1)] $\mbox{pd}\,(I(D))=\sum\limits_{i=1}^{s}r_{D_i}+s-1,$
where $r_{D_i}=\mbox{max}\,\{d(x_{i_j})+d(y_{i_{k_j}})\mid x_{i_j}\in X_i\}-2$,
\item[(2)]  $\!\sum\limits_{i=1}^{s}b_i-1\leq\mbox{pd}\,(I(D)^t)\leq|V|-\!s-1$ for $t\geq 1$, where $V$ is the  vertex set of $D$.
\end{itemize}
Furthermore, $\mbox{pd}\,(I(D)^t)$ attains this upper bound if  every  $D_{i}$ is a complete bipartite graph.
\end{Theorem}

\begin{Corollary}
 Let $D=(V, E, w)$  be a weighted oriented graph as the above theorem. If there exists some $t_0$ such that $\mbox{pd}\,(I(D)^{t_0})=|V|-s-1$, then
$$\mbox{pd}\,(I(D)^t)=|V|-s-1 \ \ \text{for any}\ \  t\geq t_0.$$
\end{Corollary}

\begin{Theorem}
Let $D=(V,E,w)$ be a weighted oriented  bipartite  graph as Theorem $1.1$. Then
 \begin{itemize}
\item[(1)] $\mbox{reg}\,(I(D))=\sum\limits_{x\in V}w(x)-|V|+s+1$,
\item[(2)] $\mbox{reg}\,(I(D)^{t})=\mbox{reg}\,(I(D))+(t-1)(w+1)$ for any $t\geq 1$,
\end{itemize}
where $w=\mbox{max}\,\{w(x)\mid x \in V \}$.
\end{Theorem}

Our paper is organized as follows. In section $2$, we recall some
definitions and basic facts used in this paper.
In section $3$, we give precise formulas for
 projective dimension and  regularity of the edge
ideal of the disjoint union of some weighted oriented gap-free bipartite graphs.
In section $4$, we provide exact formulas for regularity and give upper and lower bounds of projective dimension
of higher powers of  such an edge ideal. Moreover, we give some examples to show that regularity and projective  dimension  of powers of edge ideal of the disjoint union of some weighted oriented gap-free bipartite graphs  are related to direction selection.

For all unexplained terminology and additional information, we refer to \cite{JG} (for the theory
of digraphs), \cite{BM} (for graph theory), and \cite{BH,HH2} (for the theory of edge ideals of graphs and
monomial ideals).  We greatfully acknowledge the use of  computer algebra system CoCoA (\cite{Co}) for our experiments.

\medskip
\section{Preliminaries }

In this section, we gather together needed  definitions and basic facts, which will
be used throughout this paper. However, for more details, we refer the reader to \cite{B2,BH,GBSVV,HH2,JG,PJS,PRT,Z1,Z3,Z4}.

For any homogeneous ideal $I$ of the polynomial ring  $S=k[x_{1},\dots,x_{n}]$, there exists a {\em graded
minimal finite free resolution}

$$0\rightarrow \bigoplus\limits_{j}S(-j)^{\beta_{p,j}(I)}\rightarrow \bigoplus\limits_{j}S(-j)^{\beta_{p-1,j}(I)}\rightarrow \cdots\rightarrow \bigoplus\limits_{j}S(-j)^{\beta_{0,j}(I)}\rightarrow I\rightarrow 0,$$
where the maps are exact, $p\leq n$, and $S(-j)$ is an $S$-module obtained by shifting
the degrees of $S$ by $j$. The number
$\beta_{i,j}(I)$, the $(i,j)$-th graded Betti number of $I$, is
an invariant of $I$ that equals the number of minimal generators of degree $j$ in the
$i$th syzygy module of $I$.
Of particular interests are the following invariants which measure the ¡°size¡± of the minimal graded
free resolution of $I$.
The projective dimension of $I$, denoted pd\,$(I)$, is defined to be
$$\mbox{pd}\,(I):=\mbox{max}\,\{i\ |\ \beta_{i,j}(I)\neq 0\}.$$
The regularity of $I$, denoted $\mbox{reg}\,(I)$, is defined by
$$\mbox{reg}\,(I):=\mbox{max}\,\{j-i\ |\ \beta_{i,j}(I)\neq 0\}.$$

\medskip
The following lemma is often used in this article.
\begin{Lemma}
\label{lem1}(\cite[Lemma 1.3]{HTT}) Let  $I$ be a proper non-zero homogeneous ideal in $S$. Then
\begin{itemize}
\item[(1)] $\mbox{pd}\,(I)=\mbox{pd}\,(S/I)-1$,
\item[(2)] $\mbox{reg}\, (I)=\mbox{reg}\,(S/I)+1$.
\end{itemize}
\end{Lemma}

\medskip
Let  $u\in S$ be a monomial, we set $\mbox{supp}(u)=\{x_i: x_i|u\}$. Let $I$  be a monomial ideal and $\mathcal{G}(I)=\{u_1,\ldots,u_m\}$ denote the unique minimal set
of monomial generators of  $I$, we set $\mbox{supp}(I)=\bigcup\limits_{i=1}^{m}\mbox{supp}(u_i)$. The following two lemmas are well known.
 \begin{Lemma}
\label{lem2}(\cite[Lemma 2.5]{HTT})
Let $S_{1}=k[x_{1},\dots,x_{m}]$ and $S_{2}=k[x_{m+1},\dots,x_{n}]$ be two polynomial rings, $I\subseteq S_{1}$ and
$J\subseteq S_{2}$ be two non-zero homogeneous  ideals. Then
\begin{itemize}
\item[(1)]$\mbox{pd}\,(I+J)=\mbox{pd}\,(I)+\mbox{pd}\,(J)+1$,
\item[(2)] $\mbox{reg}\,(I+J)=\mbox{reg}\,(I)+\mbox{reg}\,(J)-1$.
\end{itemize}
\end{Lemma}

 \begin{Lemma}\label{lem3}
Let  $I, J$ be two monomial ideals  such that $\mbox{supp}\,(J)\cap \mbox{supp}\,(I)=\emptyset$. Then
\begin{itemize}
\item[(1)]$\mbox{pd}\,(JI)=\mbox{pd}\,(I)+\mbox{pd}\,(J)$,
\item[(2)]$\mbox{reg}\,(JI)=\mbox{reg}\,(I)+\mbox{reg}\,(J)$.
\end{itemize}
\end{Lemma}

\medskip
The following two lemmas can be used for computing projective dimension and regularity of an ideal.
\begin{Lemma}
\label{lem4}(\cite[Corollary 2.12]{HTT})  Let $I\subset S$ be a monomial ideal, let $f$ be a monomial of degree $k$. The following results hold.
\begin{itemize}
\item[(1)] $\mbox{pd}\,(I)\geq\mbox{pd}\,((I:f))$,
\item[(2)] If $k=1$, then $\mbox{reg}\,(I)=\mbox{reg}\,((I:f))+1$ or $\mbox{reg}\,(I)=\mbox{reg}\,((I,f))$.
\end{itemize}
\end{Lemma}

\begin{Lemma}
\label{lem5}(\cite[Lemma 1.1 and Lemma 1.2]{HTT})  Let\ \ $0\rightarrow A \rightarrow  B \rightarrow  C \rightarrow 0$\ \  be a short exact sequence of finitely generated graded $S$-modules.
Then
\begin{itemize}
\item[(1)]$\mbox{reg}\,(B)\leq \mbox{max}\, \{\mbox{reg}\,(A),\mbox{reg}\,(C)\}$, the equality holds  if $\mbox{reg}\,(A)-1\neq \mbox{reg}\,(C)$,
\item[(2)] $\mbox{pd}\,(B)\leq \mbox{max}\, \{\mbox{pd}\,(A),\mbox{pd}\,(C)\}$, the equality holds  if $ \mbox{pd}\,(C)\neq\mbox{pd}\,(A)+1$,
\item[(3)] $\mbox{pd}\,(C)\leq \mbox{max}\, \{\mbox{pd}\,(A)+1,\mbox{pd}\,(B)\}$.
\end{itemize}
\end{Lemma}

\vspace{5mm}
\section{Projective dimension  and regularity  of  edge ideal of the disjoint union of some weighted oriented gap-free  bipartite graphs}

\vspace{5mm} In this section, we will observe some basic results concerning  the disjoint union of some weighted  oriented
gap-free  bipartite graphs and   provide  some exact formulas for  projective dimension  and regularity
of the corresponding edge ideal.

 \begin{Definition}\label{def1}
 Let $s$ be  a positive integer and $G_i=(V_i,E_i)$   a simple graph for any $1\leq i\leq s$.
 \begin{itemize}
\item[(a)] They are   disjoint if they have no vertex in common, i.e., $V_i\cap V_j=\emptyset$ for any $i\neq j$.
\item[(b)] Their union is the graph  $\bigcup\limits_{i=1}^{s}G_{i}$ with vertex set $\bigcup\limits_{i=1}^{s}V_{i}$ and edge set $\bigcup\limits_{i=1}^{s}E_{i}$.
\item[(c)] If $G_1,\ldots,G_s$ are disjoint, we refer to their union  as a disjoint union, denoted $\coprod\limits_{i=1}^{s}G_{i}$.
 \end{itemize}
\end{Definition}

\begin{Definition}\label{def2}
 Let $s$ be a positive integer and  $G_1,\ldots,G_s$ be s disjoint gap-free bipartite graphs.
Let $D=(V(D),E(D),w)$ be a weighted oriented graph satisfying $D=\coprod\limits_{i=1}^{s}D_{i}$ being the disjoint union of $D_i$, where the underlying graph of $D_i$ is $G_i$. The  orientation of $D$ is as follows: Let $X_i, Y_i$ be  bipartition of $G_i$, all edges of  $D_i$ are oriented away from $X_i$  for $1\leq i\leq s$.
\end{Definition}

\medskip
 First, we provide  some properties of gap-free bipartite graphs. We prove several lemmas that will be used repeatedly throughout this paper to get our  results.
\begin{Lemma}
\label{lem6} Let $G=(V(G),E(G))$ be a gap-free bipartite graph with  bipartition $X=\{x_1,\ldots,x_{\ell}\}$, $Y=\{y_{1},\ldots,y_{m}\}$, where  $d(x_1)\leq d(x_2)\leq \cdots \leq d(x_{\ell})$. Then
\begin{itemize}
\item[(1)] $N_G(x_1)\subseteq N_G(x_2)\subseteq \cdots \subseteq N_G(x_\ell)$,
\item[(2)] $N_G(y)=X$ for any $y\in N_G(x_1)$,
\end{itemize}
\end{Lemma}
\begin{proof}
(1) Case  $\ell=1$ is  obvious. Assume $\ell\geq 2$.   If there exists some $i\in [\ell-1]$ such that $N_G(x_i)\nsubseteq N_G(x_{i+1})$.
Choose $y\in N_G(x_i)\setminus N_G(x_{i+1})$,  for any $y'\in N_G(x_{i+1})$, we have  $x_iy, x_{i+1}y'\in E(G)$. It follows that
$y'\in N_G(x_{i})$ since  $G$ is gap-free. Hence  $N_G(x_{i+1})\subsetneq N_G(x_{i})$ for the arbitrariness of $y'$. This implies $d(x_{i+1})<d(x_i)$, contradicting with $d(x_{i})\leq d(x_{i+1})$.

(2) follows from (1).
\end{proof}

\begin{Remark}
\label{rem1}If $G$ is a gap-free bipartite graph with bipartition $X=\{x_1,\ldots,x_{\ell}\}$, $Y=\{y_{1},\ldots,y_{m}\}$. We always assume  $d(x_1)\leq d(x_2)\leq \cdots \leq d(x_{\ell})$.  Throughout this paper,  by Lemma \ref{lem6} (1), we  always suppose  $N_{G}(x_i)=\{y_1,\ldots,y_{k_i}\}$  with $1\leq k_1\leq k_2\leq\cdots\leq k_{\ell}=m$  for $1\leq i\leq \ell$.
\end{Remark}

\begin{Lemma}
\label{lem7} Let $G$ be a gap-free bipartite graph as Lemma \ref{lem6}. For any $i\in [\ell]$, let  $e=x_iy_{k_i}$, then  the connected component $G'$ of  $G\setminus e$ with $|E(G')|\geq 1$ is gap-free.
\end{Lemma}
\begin{proof}
Since $|E(G')|\geq 1$, we have $|E(G)|\geq 2$. This implies $d(x_i)\geq 2$ or  $d(y_{k_i})\geq 2$. If
$d(x_i)=1$ and  $d(y_{k_i})=1$, then  $|E(G)|=1$ because  gap-free graphs are connected, a contradiction.
By direct calculation, we obtain
$$G'=\left\{\begin{array}{ll}
G\setminus x_{i}\ \ &\text{if}\ \ d(x_i)=1, \ d(y_{k_i})\geq 2,\\
G\setminus y_{k_i}\ \ &\text{if} \ \ d(x_i)\geq 2, \ d(y_{k_i})=1,\\
G\setminus e\ \ &\text{if} \ \ d(x_i)\geq 2, \ d(y_{k_i})\geq 2.
\end{array}\right.$$

If  $d(x_i)=1$,  $d(y_{{k_i}})\geq 2$, or  $d(x_{i})\geq 2$, $d(y_{k_i})=1$, then $G'$ is gap-free because it is an induced subgraph of $G$. Otherwise, we have $G'=G\setminus e$. If $G'$  is not gap-free, then there exist $a\neq c$, $b\neq d$ such that  $x_ay_b, x_cy_d\in E(G\setminus e)$  form a gap.  Obviously  $a=i$ and  $d=k_i$, or $c=i$ and  $b=k_i$. Say $a=i$ and  $d=k_i$, this implies $b<k_i$. Therefore, $x_cy_b\in E(G')$,
  contradicting with $x_ay_b, x_cy_d\in E(G\setminus e)$ forming a gap.
\end{proof}

\begin{Lemma}\label{lem8}
 Let $G$ be a gap-free bipartite graph as Lemma \ref{lem6} and  $X=\{x_1,\ldots,x_{\ell}\}$, $Y=\{y_{1},\ldots,y_{m}\}$ be its bipartition.
 Let  $N_{G}(x_i)=\{y_1,\ldots,y_{k_i}\}$ for any $i\in [\ell]$. If there exists some  $q\in [\ell]$ such that $k_q\geq2$
 and $d(x_q)+d(y_{k_q})=\mbox{max}\,\{d(x_i)+d(y_{k_i})\mid x_i\in X\}$, then
 $$d(x_q)+d(y_{k_q-1})-1\leq\mbox{max}\,\{d(x_i)+d(y_{k_i})\mid x_i\in X\}.$$
\end{Lemma}
\begin{proof} If $d(x_q)+d(y_{k_q-1})-1>d(x_q)+d(y_{k_q})$, then $d(y_{k_q-1})>d(y_{k_q})+1$. Thus there exist two different vertices $x_{p_1}, x_{p_2}$  such that $x_{p_1}, x_{p_2}\in N_{G}(y_{k_q-1})\setminus  N_{G}(y_{k_q})$. It follows that $p_1,p_2\in  [q-1]$ and $k_{p_i}=k_q-1$  from inclusion relations of neighbourhoods of elements in  $X$. Let $p_1<p_2$, then
\begin{eqnarray*}
d(x_{p_1})+d(y_{k_{p_1}})&=&d(x_{p_1})+d(y_{k_q-1})\geq (k_q-1)+(\ell-p_1+1)\\
&\geq& (k_q-1)+(\ell-(q-2))+1\\
&\geq& d(x_q)+d(y_{k_q})+1
\end{eqnarray*}
where the last inequality holds because of $y_{k_q}\in N_{G}(x_j)$ for  $q\leq j \leq \ell$,
a contradiction.
\end{proof}

\medskip
Throughout this paper,  let $D=(V(D),E(D),w)$ be a weighted oriented gap-free bipartite graph with bipartition $X=\{x_1,\ldots,x_{\ell}\}$, $Y=\{y_{1},\ldots,y_{m}\}$, then  we always assume  $d(x_1)\leq d(x_2)\leq \cdots \leq d(x_{\ell})$ and  the  orientation of every edge in $D$ is  away from $X$. Let  $N_{D}^{+}(x_i)=\{y_1,\ldots,y_{k_i}\}$  with $1\leq k_1\leq k_2\leq\cdots\leq k_{\ell}=m$  for $1\leq i\leq \ell$.

Let $D=(V(D), E(D),w)$ be a weighted oriented graph. For $T\subset V(D)$, we define
its {\em induced  subgraph} $H=(V(H), E(H),w)$  to be a  graph with $V(H)=T$,  for any $u,v\in V(H)$, $uv\in E(H)$  if and only if $uv\in E(D)$. Obviously $H=(V(H), E(H),w)$ is a  weighted oriented graph, its orientation    is the same as  in $D$.
 For any $u\in V(H)$, if $u$ is not a source in $H$, then its weight  equals to the weight of $u$ in $D$, otherwise, its weight in $H$ equals to $1$.
For  $P\subset V(D)$, we  denote
$D\setminus P$ to be the induced subgraph of $D$ obtained by removing the vertices in $P$ and the
edges incident to these vertices. If $P=\{x\}$ consists of a  vertex, then we write $D\setminus x$ for $D\setminus \{x\}$.
For $W\subseteq E(D)$, we define $D\setminus W$ to be  a subgraph of $D$ with all edges  in $W$  deleted (but its vertices remained).  When $W=\{e\}$ consists of an  edge, we write $D\setminus e$ instead of $D\setminus \{e\}$.
For $x\in V(D)$,  we denote by $N_D^{+}(x)=\{y:xy\in E(D)\}$, $N_D^{-}(x)=\{y:yx\in E(D)\}$ and $N_D(x)=N_D^{+}(x)\cup N_D^{-}(x)$.

\begin{Lemma}\label{lem9}
Let $D=(V(D),E(D),w)$ be a weighted oriented gap-free bipartite graph  with bipartition $X=\{x_1,\ldots,x_{\ell}\}$, $Y=\{y_{1},\ldots,y_{m}\}$.
 Let  $N_{D}^{+}(x_i)=\{y_1,\ldots,y_{k_i}\}$ for $1\leq i\leq \ell$ and $e=x_iy_{k_i}$.
 If $|E(D)|\geq 2$,
then
$$\mbox{pd}\,((I(D\setminus e):x_iy_{k_i}^{w_{y_{k_i}}}))=d(x_i)+d(y_{k_i})-3.$$
\end{Lemma}
\begin{proof}
Let $p=\min\{j\,|\,y_{k_i}\in N_D^{+}{(x_j)}\}$, then $p\leq i$ and $d(y_{k_i})=\ell-p+1$. Claim: $p<\ell$ or $k_i>1$. Otherwise, $k_p=m=1$ and $\ell=1$  by  connectivity of $G$ and the choice of $p$. This implies $|E(D)|=1$, a contradiction.
By direct calculation, we get
$$(I(D\setminus e):x_iy_{k_i}^{w_{y_{k_i}}})=\left\{\begin{array}{ll}
\sum\limits_{j=1}^{m-1}(y_j^{w_{y_j}})\ \ &\text{if}\ \ p=\ell, \ m>1,\\
J_i\ \ &\text{if} \ \ p<\ell, \ k_i=1,\\
J_i+\sum\limits_{j=1}^{k_i-1}(y_j^{w_{y_j}})\ \ &\text{if} \ \ p<\ell, \ k_i>1.
\end{array}\right.$$
where the minimal  set $\mathcal {G}(J_i)$ of  monomial generators of $J_i$ is $\{x_{p},\ldots, x_{\ell}\}\setminus \{x_i\}$. By Lemma \ref{lem2} (1), we obtain
\begin{eqnarray*}
\mbox{pd}\,((I(D\setminus e)\!:\!x_iy_{k_i}^{w_{y_{k_i}}}))\!\!\!\!&=&\!\!\!\!\left\{\begin{array}{ll}
\mbox{pd}\,(\sum\limits_{j=1}^{m-1}(y_j^{w_{y_j}}))  & \text{if} \ p=\ell, \ m>1,\\
\mbox{pd}\,(J_i)\  &\text{if}  \ p<\ell,\ k_i=1,\\
\mbox{pd}\,(J_i)\!+\!\mbox{pd}\,(\sum\limits_{j=1}^{k_i-1}(y_j^{w_{y_j}}))\!+\!1  &\text{if} \ p<\ell,\ k_i>1.
\end{array}\right.\\
\!\!\!\!&=&\!\!\!\!\left\{\begin{array}{ll}
m-2\ \ &\text{if}\ \ p=\ell, \ m>1,\\
\ell-p-1\ \ &\text{if} \ \ p<\ell, \ k_i=1,\\
(\ell-p-1)+(k_i-2)+1\ \ &\text{if} \ \ p<\ell, \ k_i>1.
\end{array}\right.\\
&=&k_i+(\ell-p+1)-3\\
&=&d(x_i)+d(y_{k_i})-3.
\end{eqnarray*}
This finishes the proof.
\end{proof}

\medskip
Now, We are ready to prove  the first major result of this section.
\begin{Theorem}\label{thm1}
Let $D=(V(D),E(D),w)$ be a weighted oriented gap-free bipartite graph with bipartition $X=\{x_1,\ldots,x_{\ell}\}$, $Y=\{y_{1},\ldots,y_{m}\}$.
 Let  $N_{D}^{+}(x_i)=\{y_1,\ldots,y_{k_i}\}$ for $1\leq i\leq \ell$, then
$$\mbox{pd}\,(I(D))=\mbox{max}\,\{d(x_i)+d(y_{k_i})\mid x_i\in X\}-2.$$
\end{Theorem}
\begin{proof}
We apply induction on $|E(D)|$.  Case  $|E(D)|=1$ is  obvious. Now assume   $|E(D)|\geq 2$.
If  $D$ is a  complete bipartite graph, then  $I(D)\!=\!(x_1,\ldots,x_{\ell})(y_1^{w_{y_1}}\!\!\!,\!\ldots,y_{m}^{w_{y_m}})$ and
$d(x_i)+d(y_{k_i})=\ell+m$ for any  $i\in [\ell]$.
By Lemma \ref{lem2} (1) and Lemma \ref{lem3} (1),
we get
\begin{eqnarray*}
\mbox{pd}\,(I(D))&=&\mbox{pd}\,((x_1,\!\ldots,\!x_{\ell}))+\mbox{pd}\,((y_1^{w_{y_1}}\!\!\!,\ldots,\!y_{m}^{w_{y_m}}))\\
&=&(\ell-1)+(m-1)=\mbox{max}\,\{d(x_i)+d(y_{k_i})\mid x_i\in X\}-2.
\end{eqnarray*}

Now assume that $D$ is not a  complete bipartite graph, then $k_1<m$. Let $d(x_q)+d(y_{k_q})=\mbox{max}\,\{d(x_i)+d(y_{k_i})\mid x_i\in X\}$ for some $ q\in [\ell]$ and
$e=x_qy_{k_q}$, then  $d(x_q)\geq 2$ or $d(y_{k_q})\geq 2$ by similar arguments in Lemma \ref{lem7}.
Let $D'$ be the biggest subgraph of $D\setminus e$ without isolated vertices, where $D\setminus e$  obtained from  $D$ by deleting edge $e$, then
$$D\setminus e=\left\{\begin{array}{ll}
D'\cup \{x_{q}\} \ \ &\text{if}\ \ d(x_q)=1, \ d(y_{k_q})\geq 2\\
D'\cup \{y_{k_q}\}\ \ &\text{if} \ \ d(x_q)\geq 2, \ d(y_{k_q})=1\\
D'\ \ &\text{if} \ \ d(x_q)\geq 2, \ d(y_{k_q})\geq 2
\end{array}.\right.$$
Thus $I(D')=I(D\setminus e)$. For any  $x_i\in V(D')$, if  $d(x_q)=1$ and $d(y_{k_q})\geq 2$, then $N_{D'}^{+}(x_i)=\{y_1, \ldots, y_{k_i}\}$,   otherwise,
$N_{D'}^{+}(x_i)=\{y_1, \ldots, y_{k'_i}\}$ where  if $i=q$, then $k'_i={k_{q}}-1$, otherwise, $k'_i={k_{q}}$. Moreover, if
 $d(y_{k_q})=1$, then $q=\ell$
by inclusion relations of neighbourhoods of elements in  $X$.  Case $d(x_q)=1$ and $d(y_{k_q})\geq 2$ can be  proved  by  similar arguments as other two cases, we omit it.
It remains to be show that $\mbox{pd}\,(I(D\setminus e))\leq d(x_q)+d(y_{k_q})-2$ if
 $d(x_q)\geq 2$, $d(y_{k_q})=1$, or $d(x_q)\geq 2$, $d(y_{k_q})\geq 2$.

 For any $z\in V(D')$, let  $d_{D'}(z)$ denote the degree of $z$ in  $D'$.
We have
\begin{eqnarray*}
\hspace{3cm}d_{D'}(x_i)+d_{D'}(y_{k'_i})&=& \left\{\begin{array}{ll}
d(x_q)+d(y_{k_q-1})-1 \ &\text{if} \ i=q,\\
d(x_i)+d(y_{k_i}) \ &\text{if}  \ i\neq q.\\
\end{array}\right.\\
&\leq& d(x_q)+d(y_{k_q}),    \hspace{4.5cm} (1)
\end{eqnarray*}
where the above inequality holds by Lemma \ref{lem8}.

By Lemma \ref{lem7} and induction hypothesis, we obtain
\begin{eqnarray*}
\mbox{pd}\,(I(D\setminus e))&=&\mbox{max}\,\{d_{D'}(x_i)+d_{D'}(y_{k'_i})\mid x_i\in V(D')\}-2\\
&\leq& d(x_q)+d(y_{k_q})-2. \hspace{7.8cm} (2)
\end{eqnarray*}

First, we will prove $\mbox{pd}\,(I(D))\leq \mbox{max}\,\{d(x_i)+d(y_{k_i})\mid x_i\in X\}-2$.

Consider the  exact sequence
$$\hspace{1.0cm} 0\longrightarrow \frac{S}{(I(D\setminus e):x_qy_{k_q}^{w_{y_{k_q}}})}\stackrel{ \cdot x_{q}y_{k_q}^{w_{y_{k_q}}}} \longrightarrow \frac{S}{I(D\setminus e)}\longrightarrow \frac{S}{I(D)}\longrightarrow 0\hspace{2.5cm}(*) $$
By  Lemma \ref{lem1} (1), Lemma \ref{lem5} (3),   Lemma \ref{lem9} and  formula (2),  we obtain
\begin{eqnarray*}
\mbox{pd}\,(I(D))&\leq&\max\,\{\mbox{pd}\,((I(D\setminus e):x_qy_{k_q}^{w_{y_{k_q}}}))+1,\mbox{pd}\,(I(D\setminus e))\}\\
&\leq&\max\,\{(d(x_q)+d(y_{k_q})-3)+1, d(x_q)+d(y_{k_q})-2\}\\
&=& d(x_q)+d(y_{k_q})-2 \hspace{7.8cm} (3)
\end{eqnarray*}

To obtain the desired conclusion, it is enough to prove that $\mbox{pd}\,(I(D))\geq\mbox{max}\,\{d(x_i)+d(y_{k_i})\mid x_i\in X\}-2$. We distinguish into  the following two cases:

(i) If $k_q<m$, then $q<\ell$. Let $p=\mbox{min}\,\{j\mid x_{j}\in N_D^{-}(y_{k_q+1})\}$, then $p\geq q+1$. It follows that
$$(I(D):y_{k_{q}+1}^{w_{y_{k_{q}+1}}})=(I(D''),x_p,\ldots,x_{\ell})$$
where $D''=D\setminus \{x_p,\ldots,x_{\ell},y_{k_{q}+1},\ldots,y_{m}\}$.

 For any $z\in V(D'')$, let  $d_{D''}(z)$ denote the degree of $z$ in  $D''$,
then  $N_{D''}^{+}(x_i)=\{y_1, \ldots, y_{k_i}\}$  and $d_{D''}(y_{k_i})=d(y_{k_i})-(\ell-p+1)$ by the choice of $p$ and  $x_j\in N_D^{-}(y_{k_i})$  for any $p\leq j\leq \ell$.
Hence
$$\mbox{max}\,\{d_{D''}(x_i)+d_{D''}(y_{k_i})\mid x_i\in V(D'')\}=d(x_q)+d(y_{k_q})-(\ell-p+1).$$
By Lemma \ref{lem2} (1), Lemma \ref{lem4} (1) and induction hypothesis, we obtain
\begin{eqnarray*}
\mbox{pd}\,(I(D))\!\!\!\!&\geq&\!\!\!\! \mbox{pd}\,((I(D):y_{k_{q}+1}^{w_{y_{k_{q}+1}}}))=\mbox{pd}\,(I(D''))+\mbox{pd}\,((x_p,\ldots,x_{\ell}))+1\\
\!\!\!\!&=&\!\!\!\!\mbox{max}\,\{d_{D''}(x_i)+d_{D''}(y_{k_i})\mid x_i\in V(D'')\}-2+(\ell-p)+1\\
\!\!\!\!&=&\!\!\!\!d(x_q)+d(y_{k_q})-(\ell-p+1)-2+(\ell-p)+1.\\
&=&d(x_q)+d(y_{k_q})-2.
\end{eqnarray*}

(ii) If $k_q=m$, then  $q\geq 2$  because of $k_1<m$. Let $p'=\mbox{max}\,\{1\leq j\leq \ell\mid k_j<m\}$, then
$$(I(D):x_{p'})=(I(D'''),y_1^{w_{y_1}},\ldots,y_{k_{p'}}^{w_{k_{p'}}}),$$
where $D'''=D\setminus \{x_1,\ldots,x_{p'},y_1,\ldots,y_{k_{p'}}\}$. For any $z\in V(D''')$,  let $d_{D'''}(z)$ denote the degree of $z$ in  $D'''$,
then  $N_{D'''}^{+}(x_i)=\{y_{k_{p'}+1}, \ldots, y_{k_i}\}$  and $d_{D'''}(x_i)=d(x_i)-k_{p'}$.
Since $N_{D}^{+}(x_{p'})=\{y_1,\ldots,y_{k_{p'}}\}$,   for $x_i\in V(D''')$, by the choice of $p'$, we have $N_{D'''}^{-}(y_{k_i})=N_{D'''}^{-}(y_{k_{p'+1}})=\{x_{p'+1},\ldots, x_{\ell}\}=N_{D}^{-}(y_{k_i})$. Hence $d_{D'''}(y_{k_i})=d(y_{k_i})$.
It follows that
\begin{eqnarray*}
\mbox{max}\,\{d_{D'''}(x_i)+d_{D'''}(y_{k_i})\mid x_i\in V(D''')\}\!\!\!&=&\!\!\!\mbox{max}\,\{d(x_i)-k_{p'}+d(y_{k_i})\mid x_i\in V(D''')\}\\
\!\!\!&=&\!\!\!d(x_q)+d(y_{k_q})-k_{p'}
\end{eqnarray*}
Therefore,  by Lemma \ref{lem2} (1), Lemma \ref{lem4} (1) and induction hypothesis, we get
\begin{eqnarray*}
\mbox{pd}\,(I(D))&\geq& \mbox{pd}\,((I(D):x_{p}))
=\mbox{pd}\,((I(D'''))+\mbox{pd}\,((y_1^{w_{y_1}},\ldots,y_{k_{p'}}^{w_{k_{p'}}}))+1\\
&=&\mbox{max}\,\{d_{D'''}(x_i)+d_{D'''}(y_{k_i})\mid x_i\in V(D''')\}-2+(k_{p'}-1)+1 \\
&=&d(x_q)+d(y_{k_q})-2.
\end{eqnarray*}
In short,
\[\mbox{pd}\,(I(D))\geq d(x_q)+d(y_{k_q})-2=\mbox{max}\,\{d(x_i)+d(y_{k_i})\mid x_i\in X\}-2.\eqno (4)
\]
By formulas (3) and (4), we have
\[\mbox{pd}\,(I(D))=\mbox{max}\,\{d(x_i)+d(y_{k_i})\mid x_i\in X\}-2.
\]
The proof is completed.
\end{proof}

\medskip
\begin{Corollary}\label{cor1}
Let $D=(V(D),E(D),w)$ be a weighted oriented  graph as  Theorem \ref{thm1}. Then
$$\mbox{max}\,\{\ell-1,m-1\} \leq\mbox{pd}\,(I(D))\leq |V(D)|-2.$$
Furthermore, $\mbox{pd}\,(I(D))$ attains this upper bound if and only if $D$ is a  complete bipartite graph.
\end{Corollary}
\begin{proof} From Lemma \ref{lem6} (1), we have $d(y_{k_1})=\ell$, $d(x_{\ell})=m$. It follows that $\mbox{pd}\,(I(D))\geq \mbox{max}\,\{\ell-1,m-1\}$
 by  the above theorem. On the other hand, notice that $|V(D)|=\ell+m$, $d(x_i)\leq m$ and $d(y_{k_i})\leq \ell$ for any $i\in [\ell]$, thus, by
 Theorem \ref{thm1}, we have $\mbox{pd}\,(I(D))\leq |V(D)|-2$, where the equality holds  if and only if there exists some $q\in [\ell]$ such that
 $d(x_q)=m$ and $d(y_{k_q})=\ell$, if and only if $k_q=m$ and  $x_1\in N_D^{-}(y_m)$, if and only if  $d(x_1)=m$, if and only if  $D$ is a complete bipartite graph.
\end{proof}

\medskip
As a consequence of Theorem \ref{thm1} and  Corollary \ref{cor1}, one has the following result.
\begin{Corollary}\label{cor2}
Let $s$ be a positive integer and  $D=\coprod\limits_{i=1}^{s}D_{i}$ be the disjoint union of $s$ weighted oriented gap-free bipartite graphs $D_i$ with bipartition $X_i,Y_i$ for $1\leq i\leq s$.
Let $b_i=\mbox{max}\,\{|X_i|,|Y_i|\}$ and  $N_{D_i}^{+}(x_{i_j})=\{y_{i_1},\ldots,y_{i_{k_j}}\}$ for any $x_{i_j}\in X_i$. If the orientation of $D$ is
as  Definition \ref{def2}, then
\begin{itemize}
\item[(1)] $\mbox{pd}\,(I(D))=\sum\limits_{i=1}^{s}r_{D_i}+s-1,$
where $r_{D_i}=\mbox{max}\,\{d(x_{i_j})+d(y_{i_{k_j}})\mid x_{i_j}\in X_i\}-2$,
\item[(2)] $\sum\limits_{i=1}^{s}b_{i}-1\leq\mbox{pd}\,(I(D))\leq|V|-s-1$, where $V$ is  the vertex set of $D$.
\end{itemize}
Furthermore, $\mbox{pd}\,(I(D))$ attains this upper bound if and only if every  $D_{i}$ is a complete bipartite graph.
\end{Corollary}

\medskip
Now, we are  ready to prove another major result of this section.
\begin{Theorem}\label{thm2}
Let $D=(V(D),E(D),w)$ be a weighted oriented gap-free bipartite graph. Then
$$\mbox{reg}\,(I(D))=\sum\limits_{x\in V(D)}w(x)-|V(D)|+2.$$
\end{Theorem}
\begin{proof}
Let $X=\{x_1,\ldots,x_{\ell}\}$, $Y=\{y_{1},\ldots,y_{m}\}$ be two partitions of $D$.
We apply induction on $m$. Case $m=1$ follows from \cite[Theorem 3.1]{Z3}. Assume $m\geq 2$.
 By Lemma \ref{lem6} (2), $N_D^{-}(y_1)=X$, thus
$$(I(D):y_1^{w_{y_1}})=(x_1,\ldots,x_{\ell}).$$
It follows that
\begin{eqnarray*}
\mbox{reg}\,((I(D):y_1^{w_{y_1}})(-w_{y_1}))&=&\mbox{reg}\,((x_1,\ldots,x_{\ell}))+w_{y_1}=1+w_{y_1}\\
&=&\sum\limits_{x\in V(D)}w(x)-|V(D)|+2+(m-1)-\sum\limits_{i=2}^{m}w_{y_i}\\
&\leq&\sum\limits_{x\in V(D)}w(x)-|V(D)|+2. \hspace{3.2cm} (1)
\end{eqnarray*}

Note that $(I(D),y_1^{w_{y_1}})=(I(D\setminus P),y_1^{w_{y_1}})$, where $P=\{y_1\}\cup \{x_i\in X \mid d(x_i)=1\}$.
By Lemma  \ref{lem2} (2)  and induction hypothesis, we obtain
\begin{eqnarray*}
\mbox{reg}\,((I(D),y_1^{w_{y_1}}))&=&\mbox{reg}\,(I(D\setminus P))+\mbox{reg}\,((y_1^{w_{y_1}}))-1\\
&=&\mbox{reg}\,(I(D\setminus P))+w_{y_1}-1\\
&=&\!\!\!\!\sum\limits_{x\in V(D\setminus P)}w(x)-|V(D\setminus P)|+2+w_{y_1}-1\\
&=&\!\!\!\!\sum\limits_{x\in V(D)}w(x)-(|P|-1)-(|V(D)|-|P|)+2-1\\
&=&\!\!\!\!\sum\limits_{x\in V(D)}w(x)-|V(D)|+2. \hspace{4.8cm} (2)
\end{eqnarray*}
Consider the  exact sequence
$$0\longrightarrow \frac{S}{(I(D):y_1^{w_{y_1}})}(-w_{y_1})\stackrel{ \cdot y_1^{w_{y_1}}} \longrightarrow \frac{S}{I(D)}\longrightarrow \frac{S}{(I(D),y_1^{w_{y_1}})}\longrightarrow 0,$$
By  Lemma \ref{lem1} (2), Lemma \ref{lem5} (1)  and formulas (1) and (2),  we obtain
\[
\mbox{reg}\,(I(D))=\mbox{reg}\,((I(D),y_1^{w_{y_1}}))=\sum\limits_{x\in V(D)}w(x)-|V(D)|+2.
\]
\end{proof}

\medskip
From  Theorem \ref{thm1} and Lemma \ref{lem2} (2), we have
\begin{Corollary}\label{cor3}
Let $D=(V(D),E(D),w)$ be a weighted oriented bipartite graph as  Corollary \ref{cor2}, then
$$\mbox{reg}\,(I(D))=\sum\limits_{x\in V(D)}w(x)-|V(D)|+s+1.$$
\end{Corollary}

\vspace{5mm}
\section{Regularity  and projective dimension of  powers of edge ideal of the disjoint union of some weighted oriented gap-free  bipartite graphs}
\vspace{5mm} In this section, we consider   regularity  and projective dimension  of  powers of edge ideal of  the disjoint union   of some   weighted  oriented gap-free  bipartite graphs. We will provide  some exact formulas for  regularity and give  upper and lower bounds for  projective dimension of these ideals. Meanwhile, we will give some examples to show that these formulas are related to direction selection.
We need the following lemma.

\begin{Lemma}
\label{lem10} Let $D=(V(D),E(D),w)$ be a weighted oriented gap-free bipartite graph with  bipartitions $X=\{x_1,\ldots,x_{\ell}\}$, $Y=\{y_{1},\ldots,y_{m}\}$.
 Then for any $y\in N_D^{+}(x_1)$,
$$(I(D)^{t}:x_1y^{w_y})=I(D)^{t-1} \text{ \ for  all}\ t\geq 2.$$
\end{Lemma}
\begin{proof}
For  any monomial $f\in \mathcal{G}(I(D)^{t}:x_1y^{w_y})$, we have $fx_1y^{w_y}\in I(D)^{t}$. Write $fx_1y^{w_y}$ as
$fx_1y^{w_y}=e_{i1}e_{i2}\ldots e_{it}h$, where $h$ is a  monomial and  $e_{ij}=x_{ij}y_{k}^{w_{y_{k}}}$ such that $x_{ij}y_{k}\in E(D)$. We consider the following two cases:

(i) If $x_1y^{w_y}$  is not a factor of  $e_{i1}e_{i2}\ldots e_{it}$, then $x_1$ or $y^{w_y}$ divides $h$. We may assume $x_1$ divides $h$,  there exist at most one $j\in [t]$  such that $y^{w_y}$ divides
$e_{ij}$. If $y^{w_y}\mid e_{ij}$, then $f\in I(D)^{t-1}$. Otherwise,  $f\in I(D)^{t}$.

(ii) If $x_1y^{w_y}$ is  a factor of  $e_{i1}e_{i2}\ldots e_{it}$ and $x_1y^{w_y}\neq e_{ij}$ for any $ j\in [t]$, then there exist $ k,r\in [t]$ such that $e_{ik}\neq e_{ir}$  and $x_1$ (resp. $y^{w_y}$) is  a factor of some $e_{ik}$ (resp. $e_{ir}$).
Let's say  $e_{ik}=x_1y_p^{w_{y_p}}$, $e_{ir}=x_qy^{w_{y}}$ for   $x_q\in X$, $y_p\in Y$.
 By lemma \ref{lem6} (1),  $N_D^{-}(y_p)=X$, thus $x_qy_p\in E(D)$. Hence
  $f\in I(D)^{t-1}$.
The proof is completed.
\end{proof}

\begin{Corollary}
\label{cor4}
 Let $D=(V,E,w)$ be a weighted oriented bipartite graph as  Corollary \ref{cor2}.
 Then for any $i\in [s]$, $y\in N_D^{+}(x_{i_1})$,
 $$(I(D)^{t}:x_{i_1}y^{w_y})=I(D)^{t-1}\ \text{for all}\ \ t\geq 2.$$
\end{Corollary}
\begin{proof}
Case  $s=1$ follows  from lemma \ref{lem10}. Suppose $s\geq 2$.  For convenience, let $i=1$, then $I(D)=I(D_1)+J$, where
$J=I(D_2)+I(D_3)+\cdots+I(D_s)$. It follows that $I(D)^t=(I(D_1)+J)^t=I(D_1)^t+I(D_1)^{t-1}J+\cdots+I(D_1)J^{t-1}+J^t$. By lemma \ref{lem10}, we obtain
\begin{eqnarray*}
(I(D)^{t}:x_{1_1}y^{w_y})&=&((I(D_1)^t+I(D_1)^{t-1}J+\cdots+I(D_1)J^{t-1}+J^t):x_{1_1}y^{w_y})\\
&=&I(D_1)^{t-1}+I(D_1)^{t-2}J+\cdots+I(D_1)J^{t-2}+J^{t-1}+J^{t}\\
&=&I(D)^{t-1}.
\end{eqnarray*}
\end{proof}

\medskip
The following lemma can be shown by similar arguments as in \cite[Theorem 4.2, Theorem 4.4]{Z6}, we omit its proof.
\begin{Lemma}\label{lem11}
Let $D=(V,E,w)$ be a weighted oriented bipartite graph as  Corollary \ref{cor2}. If  each $D_i$ is a star graph, then
 for any $t\geq 1$,
\begin{itemize}
\item[(1)] $\mbox{reg}\,(I(D)^t)=\sum\limits_{x\in V}w(x)-|V|+s+1+(t-1)(w+1)$,\\
 where $w=\mbox{max}\,\{w(x)\mid x\in V\}$,
\item[(2)]$\mbox{pd}\,(I(D)^t)=|V|-s-1$.
\end{itemize}

\end{Lemma}

\medskip
Now, we are ready to prove  the first major result of this section.
\begin{Theorem}\label{thm3}
Let $s$ be a positive integer and  $D=\coprod\limits_{i=1}^{s}D_{i}$ be the disjoint union of $s$ weighted oriented gap-free bipartite graphs $D_i$ with bipartition $X_i,Y_i$ for $1\leq i\leq s$.
 If the orientation of $D$ is as  Definition \ref{def2}, then
\begin{itemize}
\item[(1)] $\mbox{reg}\,(I(D)^t)\leq\sum\limits_{x\in V(D)}w(x)-|V(D)|+s+1+(t-1)(w+1)$,
\item[(2)]$\mbox{pd}\,(I(D)^t)\leq|V(D)|-s-1$,
\end{itemize}
where  $V(D)$ is the vertex set of $D$ and $w=\mbox{max}\,\{w(x)\mid x\in V(D)\}$.

Furthermore, $\mbox{pd}\,(I(D)^t)$ attains this upper bound if each $D_{i}$ is a  complete bipartite graph.
\end{Theorem}
\begin{proof} If every $D_i$ is a  star graph, then the results hold by Lemma \ref{lem11}. Otherwise, there exists some $D_i$ being not a star graph. Say $i=1$
and $D_1$  has bipartition  $X_1=\{x_1,\ldots,x_{\ell}\}$,  $Y_1=\{y_{1},\ldots,y_{m}\}$. Thus $\ell,m\geq 2$.
We apply induction on  $t$ and $|V(D)|$. Case $t=1$ follows from  Corollary \ref{cor2} and Corollary \ref{cor3}. Now assume that $t\geq 2$.
 For any $p\in [\ell]$, let $N_D^{+}(x_p)=\{y_1,\ldots,y_{k_p}\}$, then
we have the  short exact sequence
$$ 0  \longrightarrow \frac{S}{(I(D)^t:y_1^{w_{y_1}})}(-w_{y_1}) \stackrel{ \cdot y_1^{w_{y_1}}} \longrightarrow  \frac{S}{I(D)^t} \longrightarrow \frac{S}{(I(D)^t,y_1^{w_{y_1}})}  \longrightarrow  0. \eqno(\ddag)$$
Let $P=\{y_1\}\cup\{x_p\in X_1 \mid d(x_p)=1\}$, then $(I(D)^t, y_1^{w_{y_1}})=(I(D\setminus P)^t, y_1^{w_{y_1}})$ and $x_{\ell}\notin P$. It follows that $D\setminus P$ has $s$ connected components. By Lemma \ref{lem2} and induction hypothesis on $|V(D)|$, we obtain
\begin{eqnarray*}
\!\!& &\!\!\mbox{reg}\,((I(D)^t\!,\!y_1^{w_{y_1}}))\!=\!\mbox{reg}\,((I(D\setminus P)^t\!, \!y_1^{w_{y_1}}))\!=\!\mbox{reg}\,((I(D\setminus P)^t)+\mbox{reg}\, ((y_1^{w_{y_1}}))-1\\
\!\!&=&\!\!\sum\limits_{x\in V(D\setminus P)}w(x)-|V(D\setminus P)|+s+1+(t-1)(w'+1)+w_{y_1}-1\\
\!\!&=&\!\!(\sum\limits_{x\in V(D)}w(x)-(|P|-1))-(|V(D)|-|P|)+s+1+(t-1)(w'+1)-1\\
\!\!&\leq&\!\!\sum\limits_{x\in V(D)}w(x)-|V(D)|+s+1+(t-1)(w+1),   \hspace{5.2cm} (1)\\
\!\!& &\!\!\mbox{pd}\,((I(D)^t, y_1^{w_{y_1}}))=\mbox{pd}\,((I(D\setminus P)^t, y_1^{w_{y_1}}))=\mbox{pd}\,((I(D\setminus P)^t)+1\\
&\leq&(|V(D)|-|P|)-s-1+1\leq |V(D)|-s-1,
\end{eqnarray*}
where $w'=\mbox{max}\,\{w(x)\mid x\in V(D\setminus P)\}$,  the first  inequality holds by the choice of $w$ and $\mbox{pd}\,((I(D)^t, y_1^{w_{y_1}}))$ attains this upper bound if   every  $D_{i}$  is a  complete bipartite graph.

Let  $K_0=(I(D)^t:y_1^{w_{y_1}})$ and  $K_p=((I(D)^t:y_1^{w_{y_1}}),x_1,\ldots,x_{p})$ for any $ p\in [\ell]$, then we get the following short exact sequences
\begin{gather*}
\hspace{2cm}\begin{matrix}
 0 & \longrightarrow & \frac{S}{(K_0:x_1)}(-1) & \stackrel{ {\cdot x_1}} \longrightarrow & \frac{S}{K_0} &\longrightarrow & \frac{S}{K_1} & \longrightarrow & 0  &\\
  0 & \longrightarrow & \frac{S}{(K_1:x_2)}(-1) & \stackrel{ {\cdot x_2}} \longrightarrow & \frac{S}{K_1} &\longrightarrow & \frac{S}{K_2} & \longrightarrow & 0  &\hspace{0.5cm} (\ddagger\ddagger)\\
 \\
      &  &\vdots&  &\vdots&  &\vdots&  & &\hspace{4cm} \\
 0 & \longrightarrow & \frac{S}{(K_{\ell-1}:x_{\ell})}(-1) & \stackrel{ {\cdot x_{\ell}}} \longrightarrow & \frac{S}{K_{\ell-1}} &\longrightarrow & \frac{S}{K_{\ell}} & \longrightarrow & 0
 \end{matrix}
\end{gather*}
Hence, in order to compute $\mbox{reg}\,(K_0)$ and $\mbox{pd}\,(K_0)$, we need to compute  $\mbox{reg}\,(K_{\ell})$, $\mbox{pd}\,(K_{\ell})$, $\mbox{reg}\,((K_p:x_{p+1})(-1))$ and $\mbox{pd}\,((K_p:x_{p+1})(-1))$ for any $0\leq p\leq \ell-1$.
 We distinguish into the following two steps:

Step 1. If $s=1$, then $K_{\ell}=((I(D)^t:y_1^{w_{y_1}}),x_1,\ldots,x_{\ell})=(x_1,\ldots,x_{\ell})$. In this case, we have
\[\mbox{reg}\,(K_{\ell})\!=\!\mbox{reg}\,((x_1,\ldots,x_{\ell}))=1\leq\!\!\!\!\sum\limits_{x\in V(D)}\!\!\!\!w(x)\!-\!|V(D)|\!+\!s\!+\!1\!+\!(t-1)(w+1)\!-\!w_{y_1}, \]
\[\mbox{pd}\,(K_{\ell})\!=\!\mbox{pd}\,((x_1,\ldots,x_{\ell}))=\ell-1\leq |V(D)|-2.
\]

If $s\geq2$, then $K_{\ell}=((I(D)^t:y_1^{w_{y_1}}),x_1,\ldots,x_{\ell})=(I(\coprod\limits_{i=2}^{s}D_{i})^t, x_1,\ldots,x_{\ell})$.
By  Lemma \ref{lem2} and induction hypothesis on $|V(D)|$, we obtain
\begin{eqnarray*}
\mbox{reg}\,(K_{\ell})&=&\mbox{reg}\,((I(\coprod\limits_{i=2}^{s}D_{i})^t, x_1,\ldots,x_{\ell}))
=\mbox{reg}\,(I(\coprod\limits_{i=2}^{s}D_{i})^t)+\mbox{reg}\,((x_1,\ldots,x_{\ell}))-1\\
\!\!&=&(\!\!\sum\limits_{x\in V(\coprod\limits_{i=2}^{s}D_{i})}w(x)-|V(\coprod\limits_{i=2}^{s}D_{i})|+s+(t-1)(w''+1))+1-1\\
\!\!&\leq&\!\!\sum\limits_{x\in V(D)}w(x)-|V(D)|+s+1+(t-1)(w+1)-w_{y_1},   \\
\mbox{pd}\,(K_{\ell})&=&\mbox{pd}\,(I(\coprod\limits_{i=2}^{s}D_{i})^t)+\sum\limits_{j=1}^{\ell}\mbox{pd}\,((x_j))+\ell\\
\!\!&\leq&\!\!|V(\coprod\limits_{i=2}^{s}D_{i})|-(s-1)-1+\ell\leq |V(D)|-s-1,
\end{eqnarray*}
where $w''=\mbox{max}\,\{w(x)\mid x\in V(\coprod\limits_{i=2}^{s}D_{i})\}$,  the first  inequality holds by the choice of $w$.

In brief, we have
\begin{eqnarray*}
\mbox{reg}\,(K_{\ell})&\leq&\sum\limits_{x\in V(D)}w(x)-|V(D)|+s+1+(t-1)(w+1)-w_{y_1},\\
\mbox{pd}\,((K_{\ell})&\leq& |V(D)|-s-1.\hspace{7.8cm} (2)
\end{eqnarray*}
Step 2. Let $V_0=\emptyset$ and $V_p=\{x_1,\ldots,x_p\}$ for any $ p \in [\ell]$, then  $D\setminus V_{p}$ has $s$ connected components for $0\leq p \leq \ell-1$. Thus, by Corollary \ref{cor4},  we have
$$(K_0:x_1)=(I(D\setminus V_0 )^t:x_1y_1^{w_{y_1}})=I(D\setminus V_0 )^{t-1}$$
and
$$(K_p:x_{p+1})=((I(D\setminus V_p )^t:x_{p+1}y_1^{w_{y_1}}),x_1,\ldots,x_{p})=(I(D\setminus V_p )^{t-1},x_1,\ldots,x_{p}).$$
It follows that $\mbox{reg}\,((K_p:x_{p+1}))=\mbox{reg}\,(I(D\setminus V_p )^{t-1})$, $\mbox{pd}\,((K_p:x_{p+1}))=\mbox{pd}\,(I(D\setminus V_p )^{t-1})+p$ by  Lemma \ref{lem2}.
Therefore, by  induction hypothesis on $t$ and $|V(D)|$, we obtain
\begin{eqnarray*}
\!\!& &\!\!\mbox{reg}\,((K_p:x_{p+1})(-1))\!=\!\mbox{reg}\,((K_p:x_{p+1}))+1=\mbox{reg}\,(I(D\setminus V_p )^{t-1})+1\\
\!\!&=&\!\!\sum\limits_{x\in V(D\setminus V_p)}w(x)-|V(D\setminus V_p)|+s+1+(t-2)(w+1)+1\\
\!\!&=&\!\!(\sum\limits_{x\in V(D)}w(x)-|V_p|)-(|V(D)|-|V_p|)+s+1+(t-1)(w+1)-w\\
\!\!&\leq&\!\!\sum\limits_{x\in V(D)}w(x)-|V(D)|+s+1+(t-1)(w+1)-w_{y_1}, \\
\!\!& &\!\!\mbox{pd}\,((K_p:x_{p+1}))\!=\mbox{pd}\,(I(D\setminus V_p )^{t-1})+p\\
&\leq&|V(D\setminus V_p)|-s-1+p=|V(D)|-s-1, \hspace{5.3cm} (3)
\end{eqnarray*}
where the first  inequality holds by the choice of $w$ and $\mbox{pd}\,((K_p:x_{p+1}))$ attains this upper bound if  every $D_{i}$ is a  complete bipartite graph.

Therefore,  by Lemma \ref{lem1}, (1) and (2) of  Lemma \ref{lem5}, formulas (2), (3) and the short exact sequences ($\ddagger\ddagger$), we obtain
\begin{eqnarray*}
\mbox{reg}\,(K_0)\!\!&=&\!\!\mbox{reg}\,(I(D)^t:y_1^{w_{y_1}})
\leq\!\!\!\sum\limits_{x\in V(D)}\!\!\!w(x)-|V(D)|+s+1+(t-1)(w+1)-w_{y_1},\\
\mbox{pd}\,(K_0)\!\!&=&\mbox{pd}\,(I(D)^t:y_1^{w_{y_1}})\!\!\leq|V(D)|-s-1, \hspace{5.0cm} (4)
\end{eqnarray*}
and $\mbox{pd}\,(K_0)$ attains this upper bound if  every $D_{i}$ is a  complete bipartite graph.

Finally, by Lemma \ref{lem1}, (1) and (2) of  Lemma \ref{lem5}, formulas (1), (4) and the short exact sequence ($\ddagger$), we obtain
\begin{eqnarray*}
\mbox{reg}\,(I(D)^t)&\leq&\sum\limits_{x\in V(D)}w(x)-|V(D)|+s+1+(t-1)(w+1),\\
\mbox{pd}\,(I(D)^t)&\leq & |V(D)|-s-1.
\end{eqnarray*}
Furthermore,  $\mbox{pd}\,(I(D)^t)$ attains this upper bound if  every $D_{i}$ is a  complete bipartite graph.
We complete the proof.
\end{proof}

\medskip
\begin{Corollary}\label{cor5}
Let $D=(V,E,w)$ be a weighted oriented bipartite graph as  Theorem \ref{thm3}. If $b_i=\mbox{max}\,\{|X_i|,|Y_i|\}$, then for any $t\geq 1$
\begin{itemize}
\item[(1)] $\sum\limits_{i=1}^{s}b_i-1\leq\mbox{pd}\,(I(D)^t)\leq|V|-s-1$,
\item[(2)] $s+1\leq\mbox{depth}\,(I(D)^t)\leq |V|-\sum\limits_{i=1}^{s}b_i+1$.
\end{itemize}
Furthermore, $\mbox{pd}\,(I(D)^t)$ and $\mbox{depth}\,(I(D)^t)$  attain upper and lower bounds respectively if each  $D_{i}$ is a complete bipartite graph and  the  lower and upper bounds of $\mbox{pd}\,(I(D)^t)$ and  $\mbox{depth}\,(I(D)^t)$ are the same respectively if  each  $D_{i}$ is a star graph.
\end{Corollary}
\begin{proof} By Auslander-Buchsbaum formula and Theorem \ref{thm3}, we only need to prove $\mbox{pd}\,(I(D)^t)\geq \sum\limits_{i=1}^{s}b_i-1$ and
$\sum\limits_{i=1}^{s}b_i=|V|-s$ if  each  $D_{i}$ is a star graph.
Let $x_1$ be the smallest degree in $\bigcup\limits_{i=1}^{s}X_i$ and $y_1\in N_D^{+}(x_1)$. We apply induction on  $t$.
Case $t=1$ follows from Corollary \ref{cor2}.
Assume $t\geq 2$, by  Lemma \ref{lem4} (1),  Corollary \ref{cor4} and induction hypothesis on $t$, we obtain
\[
\mbox{pd}\,(I(D)^t)\geq\mbox{pd}\,((I(D)^t:x_1y_1^{w_{y_1}}))=\mbox{pd}\,(I(D)^{t-1})\geq\sum\limits_{i=1}^{s}b_i-1.
\]
If  each  $D_{i}$ is a star graph, then $\sum\limits_{i=1}^{s}b_i=|V|-s$.
\end{proof}

\begin{Corollary}\label{cor6}
 Let $D=(V, E, w)$  be a weighted oriented graph as Theorem \ref{thm3}. If there exists some $t_0$ such that $\mbox{pd}\,(I(D)^{t_0})=|V|-s-1$, then
\begin{itemize}
\item[(1)] $\mbox{pd}\,(I(D)^t)=|V|-s-1$ for any $t\geq t_0$,
\item[(2)] $\mbox{depth}\,(I(D)^t)=s+1$ for any $t\geq t_0$.
\end{itemize}
\end{Corollary}
\begin{proof} By Auslander-Buchsbaum formula and Theorem \ref{thm3}, we only need to prove $\mbox{pd}\,(I(D)^t)\geq|V|-s-1$ for any $t\geq t_0$.
Let $x_1$ be the smallest degree in $\bigcup\limits_{i=1}^{s}X_i$ and $y_1\in N_D^{+}(x_1)$.
If $t>t_0$, then, by  Lemma \ref{lem4} (1),  Corollary \ref{cor4} and induction hypothesis on $t$, we obtain
\[
\mbox{pd}\,(I(D)^t)\geq\mbox{pd}\,((I(D)^t:x_1y_1^{w_{y_1}}))=\mbox{pd}\,(I(D)^{t-1})\geq |V|-s-1.
\]
\end{proof}

An immediate consequence of Corollary \ref{cor6} is the following corollary.
\begin{Corollary}\label{cor7}
 Let $s$ be a positive integer and $G_1,\ldots,G_s$  be $s$ disjoint gap-free bipartite graphs. Let  $G=\coprod\limits_{i=1}^{s}G_{i}$ be  their disjoint union.
If there exists some $t_0$ such that $\mbox{pd}\,(I(G)^{t_0})=|V|-s-1$, then
\begin{itemize}
\item[(1)] $\mbox{pd}\,(I(G)^t)=|V|-s-1$ for any $t\geq t_0$,
\item[(2)] $\mbox{depth}\,(I(G)^t)=s+1$ for any $t\geq t_0$.
\end{itemize}
\end{Corollary}

\medskip
The following example shows that  if there exists some $t$ such that $\mbox{pd}\,(I(D)^t)=|V(D)|-s-1$,  we can not obtain  every $D_i$ is a  complete bipartite graph in Corollary \ref{cor6}.
\begin{Example}  \label{example1}
Let $D=D_{1}\coprod D_{2}$ be   disjoint union of two weighted oriented gap-free bipartite graphs, where $D_1$ (resp. $D_2$) is a
 digraph with $E(D_1)=\{x_1y_1, x_1y_2, x_2y_1, x_2y_2\}$ (resp. $E(D_2)=\{x_3y_3, x_4y_3, x_4y_4\}$). Thus $D_1$ (resp. $D_2$) is a complete bipartite graph
( resp. is not a complete bipartite graph).
The  weight function  of $D$ is: $w(x_1)=\cdots=w(x_4)=1$ and $w(y_1)=\cdots=w(y_4)=2$. Then the edge ideal of $D$ is
$I(D)=(x_1y_1^{2},x_1y_2^{2},x_2y_1^{2},x_2y_2^{2}, x_3y_3^{2},x_4y_3^{2},x_4y_4^{2})$. By using CoCoA, we obtain $\mbox{pd}\,(I(D)^2)=5=8-2-1$.
\end{Example}

\medskip
Next, we are  ready to prove another major result of this section.
\begin{Theorem}\label{thm4}
Let $D=(V(D),E(D),w)$ be a weighted oriented bipartite graph as  Theorem \ref{thm3}. Then  for any $t\geq 1$,
$$\mbox{reg}\,(I(D)^t)=\sum\limits_{x\in V(D)}w(x)-|V(D)|+s+1+(t-1)(w+1),$$
where $w=\mbox{max}\,\{w(x)\mid x\in V(D)\}$.
\end{Theorem}
\begin{proof} If every $D_{i}$ is a  star graph, then the  results hold by Lemma \ref{lem11}. Otherwise, there exists some  $D_{i}$ which is not a star graph. Say $i=1$ and $D_{1}$ has bipartition  $X_1=\{x_1,\ldots,x_{\ell}\}$,  $Y_1=\{y_{1},\ldots,y_{m}\}$ with
 the degree of $x_{1}$ being the smallest in $X_1$. Thus $\ell,m\geq 2$.
We apply induction on  $t$ and $|V(D)|$. Case $t=1$ follows from  Corollary \ref{cor3}. Now assume that $t\geq 2$.
By Lemma \ref{lem2} (2) and induction hypothesis on $|V(D)|$, we obtain
\begin{eqnarray*}
 \mbox{reg}\,((I(D)^t,x_1))&=&\mbox{reg}\,((I(D\setminus x_1)^t,x_1))=\mbox{reg}\,(I(D\setminus x_1)^t)+\mbox{reg}\,((x_1))-1\\
\!\!\!\!&=&\!\!\!\!\sum\limits_{x\in V(D\setminus x_1)}w(x)-|V(D\setminus x_1)|+s+1+(t-1)(w+1)\\
\!\!\!\!&=&\!\!\!\!\sum\limits_{x\in V(D)}w(x)-|V(D)|+s+1+(t-1)(w+1).\hspace{2.3cm} (1)
\end{eqnarray*}
Using  the short exact sequence
$$  0  \longrightarrow  \frac{S}{(I(D)^t:x_1)}(-1)  \stackrel{ \cdot x_1} \longrightarrow \frac{S}{I(D)^t} \longrightarrow  \frac{S}{(I(D)^t,x_1)}  \longrightarrow  0,  \eqno(\ddag)$$
we need to  calculate $\mbox{reg}\,((I(D)^t:x_1)(-1))$ in order to compute $\mbox{reg}\,(I(D)^t)$. Let
 $N_D^{+}(x_1)=\{y_1,\ldots,y_{k_1}\}$, $K_{k_1}=(I(D)^t:x_1)$ and $K_i=(I(D)^t:x_1)+\sum\limits_{j=i+1}^{k_1}(y_j^{w_{y_{j}}})$ for $0\leq i\leq {k_1-1}$, then we have the  short exact sequences
\begin{gather*}
\begin{matrix}
 0 & \longrightarrow & \frac{S}{(K_{k_1}:y_{k_1}^{w_{y_{k_1}}})}(-w_{y_{k_1}}) & \stackrel{ \cdot y_{k_1}^{w_{y_{k_1}}}} \longrightarrow & \frac{S}{K_{k_1}} &\longrightarrow & \frac{S}{K_{k_1-1}} & \longrightarrow & 0  &\\
 0 & \longrightarrow & \frac{S}{(K_{k_1-1}:y_{{k_1}-1}^{w_{y_{{k_1}-1}}})}(-w_{y_{{k_1}-1}}) & \stackrel{ \cdot y_{{k_1}-1}^{w_{y_{{k_1}-1}}}} \longrightarrow & \frac{S}{K_{k_1-1}} &\longrightarrow & \frac{S}{K_{{k_1}-2}} & \longrightarrow & 0    \\
      &  &\vdots&  &\vdots&  &\vdots&  & &\hspace{1.0cm} (\ddagger\ddagger) \\
 0 & \longrightarrow & \frac{S}{(K_{1}:y_1^{w_{y_1}})}(-w_{y_1}) & \stackrel{ \cdot y_1^{w_{y_1}}} \longrightarrow & \frac{S}{K_{1}} &\longrightarrow & \frac{S}{K_0} & \longrightarrow & 0.
 \end{matrix}
\end{gather*}
Let  $V_{k_1}=\emptyset$ and $V_{i}=\{y_{i+1},\ldots,y_{k_1}\}$ for $0\leq i\leq {k_1-1}$, then $K_{0}=I(D\setminus V_0)^t+\sum\limits_{j=1}^{k_1}(y_j^{w_{y_{j}}})$, $K_i=(I(D\setminus V_i)^t:x_1)+\sum\limits_{j=i+1}^{k_1}(y_j^{w_{y_{j}}})$  and $y_1\notin V_i$ for $i\in [{k_1}]$. It follows that $D\setminus V_i$ has $s$ connected components for  $i\in [{k_1}]$. We distinguish into the following two steps:

Step 1. We  will compute  $\mbox{reg}\,((K_{i}:y_{i}^{w_{y_{i}}})(-w_{y_{i}}))$ for $i\in [k_1]$.

Notice that $(K_{i}:y_{i}^{w_{y_{i}}})=(I(D\setminus V_i)^t:x_1y_{i}^{w_{y_{i}}})+\sum\limits_{j=i+1}^{k_1}(y_j^{w_{y_{j}}})
=I(D\setminus V_i)^{t-1}+\sum\limits_{j=i+1}^{k_1}(y_j^{w_{y_{j}}})$
by Corollary \ref{cor4}. For any $i\in [k_1]$, by Lemma \ref{lem2} (2) and induction hypothesis on $|V(D)|$, we obtain
\begin{eqnarray*}
& &\mbox{reg}\,((K_{i}:y_{i}^{w_{y_{i}}}))\!=\!\mbox{reg}\,(I(D\setminus V_i)^{t-1}+\sum\limits_{j=i+1}^{k_1}(y_j^{w_{y_{j}}}))\\
&=&\mbox{reg}\,(I(D\setminus V_i)^{t-1})+\sum\limits_{j=i+1}^{k_1}\mbox{reg}\,((y_j^{w_{y_j}}))-(k_1-i)\\
\!\!\!\!&=&(\!\!\!\!\sum\limits_{x\in V(D\setminus V_i)}\!\!\!\!\!\!w(x)-|V(D\setminus V_i)|+s+1+(t-2)(w'+1))
+\!\!\sum\limits_{j=i+1}^{k_1}\!w_{y_j}-(k_1-i)\\
\!\!\!\!&=&\!\!\!\!\sum\limits_{x\in V(D)}\!\!\!\!w(x)-(|V(D)|-(k_1-i))+s+1+(t-2)(w'+1)-(k_1-i)\\
\!\!\!\!&=&\!\!\!\!\sum\limits_{x\in V(D)}\!\!\!\!w(x)-|V(D)|+s+1+(t-2)(w'+1),
\end{eqnarray*}
where $w'=\mbox{max}\,\{w(x)\mid x\in V(D\setminus V_i)\}$.

By the choice of $w$, we have
$$\mbox{reg}\,((K_{i}:y_{i}^{w_{y_{i}}})(-w_{y_{i}}))\leq \sum\limits_{x\in V(D)}\!\!\!\!w(x)-|V(D)|+s+1+(t-1)(w+1)-1.\eqno(2)$$

Step 2. We  will compute  $\mbox{reg}\,(K_{k_1})$.

By direct calculation, we get
\[
K_0=\left\{\begin{array}{ll}
B\ \ &\text{if}\  \ k_1=m,\ s=1\\
A+B \ &\text{otherwise}\\
\end{array},\right.
\]
where $A=I(D\setminus V_0)^t=\left\{\begin{array}{ll}
I(\coprod\limits_{i=2}^{s}D_{i})^t\ \ &\text{if}\  \ k_1=m, \ s\geq2\\
I(D\setminus V_0)^t\ \ &\text{if}\  \ k_1<m\\\end{array}\right.$
 and $B=\sum\limits_{j=1}^{k_1}(y_j^{w_{y_j}})$.

By induction hypothesis on $|V(D)|$, we obtain
\[
\mbox{reg}\,(A)=\left\{\begin{array}{ll}
C+s\ \ &\text{if}\  \ k_1=m, \ s\geq 2\\
C+s+1\ \ &\text{if}\  \ k_1<m, \ s\geq 1\\
\end{array}\right.,
\]
where $C=\sum\limits_{x\in V(D\setminus V_0)}\!\!\!\!\!\!w(x)-|V(D\setminus V_0)|+(t-1)(w''+1)
=\sum\limits_{x\in V(D)}\!\!\!w(x)-|V(D)|+k_1-\sum\limits_{j=1}^{k_1}w_{y_j}$
and  $w''=\mbox{max}\,\{w(x)\mid x \in V(D\setminus V_0)\}$.
By Lemma \ref{lem2} (2), it follows that
\begin{eqnarray*}
\mbox{reg}\,(K_0)&=&\left\{\begin{array}{ll}
\mbox{reg}\,(\sum\limits_{j=1}^{m}(y_j^{w_{y_j}}))\ \ &\text{if}\  \ k_1=m, \ s=1,\\
\mbox{reg}\,(A)+\mbox{reg}\,(\sum\limits_{j=1}^{k_1}(y_j^{w_{y_j}}))-1\ \  &\text{otherwise}.\\
\end{array}\right.\\
&=&\left\{\begin{array}{ll}
\sum\limits_{j=1}^{m}w_{y_j}-(m-1)\ \ &\text{if}\  \ k_1=m, \ s=1,\\
C+s+\sum\limits_{j=1}^{m}w_{y_j}-(m-1)-1\ \  &\text{if}\  \ k_1=m, \ s\geq 2,\\
C+s+1+\sum\limits_{j=1}^{k_1}w_{y_j}-(k_1-1)-1\ \  &\text{if}\  \ k_1<m.\\
\end{array}\right.
\end{eqnarray*}
Hence, if $k_1<m$ and $w''=w$, then $\mbox{reg}\,(K_0)=\sum\limits_{x\in V(D)}\!\!\!w(x)\!-\!|V(D)|+s+1+(t-1)(w+1)$; otherwise,
$\mbox{reg}\,(K_0)=\sum\limits_{x\in V(D)}\!\!\!w(x)\!-\!|V(D)|+s+1+(t-1)(w+1)-1$.

Using  Lemma \ref{lem1} (2), Lemma \ref{lem5} (1),  and the short exact sequences ($\ddag\ddag$),  we obtain
if $k_1<m$ and $w''=w$, then $\mbox{reg}\,(K_{k_1})=\mbox{reg}\,((I(D)^t:x_1))=\sum\limits_{x\in V(D)}\!\!\!w(x)\!-\!|V(D)|+s+1+(t-1)(w+1)$;
otherwise, $\mbox{reg}\,(K_{k_1})=\mbox{reg}\,((I(D)^t:x_1))\leq \sum\limits_{x\in V(D)}\!\!\!w(x)\!-\!|V(D)|+s+1+(t-1)(w+1)-1$.

Finally,  using Lemma \ref{lem1} (2), Lemma \ref{lem4} (2), Lemma \ref{lem5} (1), Theorem \ref{thm3},  formulas  (1), (2),  the formula of $\mbox{reg}\,(K_{k_1})$ and the short exact sequences ($\ddagger$), we obtain
$$\mbox{reg}\,(I(D)^t)=\mbox{reg}\,((I(D)^t,x_1))=\sum\limits_{x\in V(D)}w(x)-|V(D)|+s+1+(t-1)(w+1).$$
We complete the proof.
\end{proof}

\medskip
From Corollary \ref{cor3} and Theorem \ref{thm4}, we have
\begin{Corollary}\label{cor8}
Let $D=(V,E,w)$ be a weighted oriented  bipartite  graph as Theorem \ref{thm4}. Then
$$\mbox{reg}\,(I(D)^{t})=\mbox{reg}\,(I(D))+(t-1)(w+1)\hspace{1cm}  \mbox{for any}\ \ t\geq 1,$$
where $w=\mbox{max}\,\{w(x)\mid x \in V\}$.
\end{Corollary}

\medskip
A {\it matching} in a graph $G$ is a subgraph consisting of pairwise disjoint edges. The
largest size of a matching in $G$ is called its {\it matching number} and denoted by $c(G)$. If the subgraph is an induced subgraph, the matching
is an {\it induced matching}. The largest size of an induced matching in $G$ is called its
{\it induced matching number} and denoted by $\nu(G)$.

As a consequence of Corollary \ref {cor5} and Theorem \ref{thm4}, we have the following corollary.
\begin{Corollary}\label{cor9}
Let $s$ be a positive integer,  $G_1,\ldots,G_s$  be $s$ disjoint gap-free bipartite graphs. Let  $G=\coprod\limits_{i=1}^{s}G_{i}$ be  their disjoint union, where  every $G_i$  has  bipartition $X_i$, $Y_i$ and $b_i=\mbox{max}\,\{|X_i|,|Y_i|\}$. If  $\nu(G)$ denotes the induced matching number of $G$,
then for any $t\geq 1$
\begin{itemize}
\item[(1)] $\mbox{reg}\,(I(G)^t)=2t+\nu(G)-1$,
\item[(2)] $\sum\limits_{i=1}^{s}b_i-1\leq\mbox{pd}\,(I(G)^t)\leq|V(G)|-s-1$,
\item[(3)] $s+1\leq\mbox{depth}\,(I(G)^t)\leq |V(G)|-\sum\limits_{i=1}^{s}b_i+1$.
\end{itemize}
 Furthermore,  $\mbox{pd}\,(I(G)^t)$ and $\mbox{depth}\,(I(G)^t)$  attain the upper and lower bounds respectively if each  $G_{i}$ is a complete bipartite graph and  the  lower and upper bounds of $\mbox{pd}\,(I(G)^t)$ and  $\mbox{depth}\,(I(G)^t)$ are the same respectively if  each  $G_{i}$ is a star graph.
\end{Corollary}
\begin{proof} By the definition of $\nu(G)$, we have $\nu(G)=s$. Let $D=\coprod\limits_{i=1}^{s}D_{i}$ be a  weighted oriented graph such that the underlying graph  of $D_i$ is $G_i$. If $w(x)=1$ for any  $x\in V(D)$, then $I(G)=I(D)$. Hence the results  follow  from Theorem \ref{thm4} and Corollary \ref{cor5}.
 \end{proof}

\medskip
It is well known that  for a simple graph $G$, we have $\mbox{reg}\,(I(G)^t)\geq 2t+\nu(G)-1$ for all $t\geq 1$ (see \cite[Theorem 4.15]{B1}). By Corollary \ref{cor9}, we obtain that  $\mbox{reg}\,(I(G)^t)$ reaches this lower bounds  if $G$ is the disjoint union of some  gap-free  bipartite graphs.

\medskip
The following three examples show that   regularity and  projection dimension of  powers of edge ideals of    graphs as Corollary \ref{cor2}, Theorem \ref{thm4} and Corollary \ref{cor5} are related to
direction selection.
\begin{Example}  \label{example2}
Let $I(D)=(x_1y_1^2,y_2x_1^2,y_1x_2^2,x_2y_2^2)$  be the edge ideal of a weighted oriented complete bipartite graph with  weight  $w(x_1)=w(x_2)=w(y_1)=w(y_2)=2$, thus we have  $w=\mbox{max}\,\{w(x)\mid x \in V(D)\}=2$. By using CoCoA, we obtain $\mbox{reg}\,(I(D))=5$, $\mbox{reg}\,(I(D)^2)=8$ and $\mbox{pd}\,(I(D))=\mbox{pd}\,(I(D)^2)=3$. But we have $\mbox{reg}\,(I(D))=\sum\limits_{x\in V(D)}w(x)-|V(D)|+2=6$, $\mbox{reg}\,(I(D)^2)=\sum\limits_{x\in V(D)}w(x)-|V(D)|+2+(w+1)=9$  by Theorem \ref{thm4} and
$\mbox{pd}\,(I(D))=\mbox{pd}\,(I(D)^2)=|V(D)|-2=2$ by Corollary \ref{cor5}.
\end{Example}

\begin{Example}  \label{example3}
Let $D=D_{1}\coprod D_{2}$ be the  disjoint union of two weighted oriented gap-free bipartite graphs, where $E(D_1)=\{y_1x_1, x_2y_1, x_2y_2\}$ and $E(D_2)=\{x_3y_3, y_3x_4, x_4y_4, x_5y_3, y_4x_5\}$. Then $r_{D_1}=d(x_1)+d(y_1)-2=1$ and $r_{D_2}=d(x_3)+d(y_3)-2=2$.
The  weight function  of $D$ is: $w(x_1)=\cdots=w(x_5)=1$ and $w(y_1)=\cdots=w(y_4)=2$. Then the edge ideal of $D$ is
$I(D)\!=\!(y_1x_1, x_2y_1^2, x_2y_2^2, x_3y_3^2, y_3x_4, x_4y_4^2,\\ x_5y_3^2, y_4x_5)$. By using CoCoA, we obtain $\mbox{reg}\,(I(D))=6$, $\mbox{reg}\,(I(D)^2)=9$ and $\mbox{pd}\,(I(D))=5$. But we have $\mbox{reg}\,(I(D))=\sum\limits_{x\in V(D)}w(x)-|V(D)|+2+1=7$, $\mbox{reg}\,(I(D)^2)=\sum\limits_{x\in V(D)}w(x)-|V(D)|+2+1+(w+1)=10$  by Theorem \ref{thm4} and
$\mbox{pd}\,(I(D))=r_{D_1}+r_{D_2}+2-1=4$ by Corollary \ref{cor2}.
\end{Example}

\begin{Example}  \label{example4}
Let $I(D)=(x_1y_1^2, y_1x_2, x_2y_2^2, x_3y_1^2, y_2x_3^2)$ be the edge ideal of a weighted oriented gap-free bipartite with weight  $w(x_1)=w(x_2)=1$ and $w(x_3)=w(y_1)=w(y_2)=2$. By using CoCoA, we obtain $\mbox{pd}\,(I(D)^3)=4$. But we have
$\mbox{pd}\,(I(D)^3)\leq|V(D)|-2=3$ by Corollary \ref{cor5}.
\end{Example}

\medskip

\hspace{-6mm} {\bf Acknowledgments}

 \vspace{3mm}
\hspace{-6mm}  This research is supported by the National Natural Science Foundation of China (No.11271275) and  by foundation of the Priority Academic Program Development of Jiangsu Higher Education Institutions.


\begin{thebibliography}{99}



\bibitem{AB} A. Alilooee and A. Banerjee, Powers of edge ideals of regularity three bipartite graphs, {\it J. Commut. Algebra}, 9 (4) (2017), 441-454.


\bibitem{ABS} A. Alilooee,  A. Banerjee and S. Selvaraja, Regularity of powers of edge ideal of unicyclic graphs,
{\it arXiv: 1702.001916V4}.

\bibitem{B1}  A. Banerjee, The regularity of powers of edge ideals, {\it J. Algebraic Combin.}, 41 (2014), 303-321.

\bibitem{B2}  A. Banerjee, Regularity of path ideals of gap free graphs,  {\it J. Pure Appl. Algebra}, 221 (2017), 2409-2419.



\bibitem{BBH1}  A. Banerjee, S. Beyarslan, and H. T. H\`a, Regularity of edge ideals and their powers,
{\it  Advances in algebra, Springer Proc. Math. Stat., Springer, Cham,}, 277 (2019), 17-52.


\bibitem{BBH2}  A. Banerjee, S. Beyarslan, and H. T.  H\`a, Regularity of powers of edge ideals: from local properties to global bounds, {\it  arXiv:1805.01434V2}.


\bibitem{BHT} S. Beyarslan, H. T. H\`a and T. N. Trung, Regularity of powers of forests and cycles, {\it J.
Algebraic Combin.}, 42 (2015),  1077-1095.


\bibitem{BM} J. A. Bondy and U. S. R. Murty, {\it Graph Theory}, Springer.com, 2008.

\bibitem{B3}  M. Brodmann,  The asymptotic nature of the analytic spread, {\it  Math. Proc. Cambridge Philos Soc.}, 86 (1979), 35-39.


\bibitem{BH} W. Bruns and J. Herzog, {\it Cohen-Macaulay rings}, Revised Edition, Cambridge University Press,
1998.


\bibitem{Co} CoCoATeam, CoCoA: a system for doing Computations in Commutative Algebra, Avaible at
http://cocoa.dima.unige.it

\bibitem{Con} A. Conca, Regularity jumps for powers of ideals,   In Commutative algebra, volume 244 of  {\it Lect. Notes
Pure Appl. Math.},  (2006), 21-32.


\bibitem{CHT} S. D. Cutkosky, J. Herzog and N. V. Trung, Asymptotic behaviour of the Castelnuovo-Mumford regularity, {\it Compos. Math.} 118 (3) (1999), 243-261.



\bibitem{FM} L. Fouli and  S. Morey,  A lower bound for depths of powers of edge ideals, {\it J. Algebr Comb}, 42 (2015), 829-848.


\bibitem{GBSVV} P. Gimenez, J. M. Bernal, A. Simis, R. H. Villarreal, and
C. E. Vivares, Monomial ideals and Cohen-Macaulay vertex-weighted digraphs,
{\it Singularities, algebraic geometry, commutative algebra, and related topics, Springer, Cham,} (2018) 491-510.


\bibitem{HTT} H. T.  H\`a, N. V. Trung, and  T. N. Trung, Depth and regularity of powers of sums of ideals, {\it Math. Z.},
282 (3-4) (2016), 819-838.

\bibitem{HH2} J. Herzog and T. Hibi,  {\it  Monomial Ideals}, New York, NY, USA: Springer-Verlag, 2011.

\bibitem{HHZ}  J. Herzog,  T. Hibi, and  X. Zheng, Monomial ideals whose powers have a linear resolution, {\it Math. Scand.},
95 (2004), 23¨C32.

\bibitem{HH3} J. Herzog and T. Hibi,  The depth of powers of an ideal, {\it J. Algebra}, 291 (2005), 534-550.



\bibitem{JG} J. B. Jensen and G. Gutin, {\it Digraphs. Theory, Algorithms and Applications}, Springer Monographs in
Mathematics, Springer, 2006.


\bibitem{JNS} A. V. Jayanthan, N. Narayanan and S. Selvaraja, Regularity of powers of bipartite graphs, {\it J.
Algebraic Combin.}, 47 (1) (2018), 17-38.

\bibitem{K} V. Kodiyalam, Asymptotic behaviour of Castelnuovo-Mumford regularity,
{\it Proc. Amer. Math. Soc.}, 128 (1999), 407-411.


\bibitem{MPV} J. Mart\'inez-Bernal, Y. Pitones and R. H. Villarreal, Minimum distance functions of graded
ideals and Reed-Muller-type codes, {\it J. Pure Appl. Algebra}, 221 (2017), 251-275.


\bibitem{MSY} M. Moghimian, S. A. Fakhari and S. Yassemi,  Regularity of powers of edge ideal of whiskered cycles, {\it Comm. Algebra}, 45(3) (2016), 1246-1259.


\bibitem{M}  S. Morey, Depths of powers of the edge ideal of a tree, {\it Comm. Algebra}, 38, (2010),  4042-4055.

\bibitem{PJS} P. Gimenez, J. Mart\'inez-Bernal, A. Simis, R. H. Villarreal and C. E. Vivares, Symbolic powers of
monomial ideals and Cohen¨CMacaulay vertex-weighted digraphs, special volume dedicated to Antonio
Campillo, Springer, to appear.

\bibitem{PS} C. Paulsen and S. Sather-Wagstaff, Edge ideals of weighted graphs, {\it J. Algebra Appl.}, 12 (5)
(2013),  1250223-1-24.

\bibitem{PRT} Y. Pitones, E. Reyes, and J. Toledo, Monomial ideals of weighted oriented graphs,
 arXiv:1710.03785.

\bibitem{RJNP}  Y. C. Ruiz, S. Jafari,  N. Nemati and B. Picone,  Regularity of bicyclic graphs and their powers,
{\it ArXiv: 1802.07202V1}.


\bibitem{TW} N. V. Trung and  H. Wang, On the asymptotic behavior of Castelnuovo-Mumford regularity, {\it J. Pure Appl.
Algebra}, 201 (2005), 42¨C48.


\bibitem{V2} Kodiyalam, Vijay, Asymptotic behaviour of Castelnuovo-Mumford regularity, {\it Proc. Am. Math. Soc.},
128 (2) (2000), 407-411.



\bibitem{Z1} Guangjun Zhu, Projective dimension and regularity of the path ideal of the line graph,  {\it J. Algebra Appl.}, 17 (4), (2018), 1850068-1-15.

\bibitem{Z2} Guangjun Zhu, Projective dimension and  regularity of  path  ideals  of  cycles, {\it J. Algebra Appl.}, 17 (10), (2018), 1850188-1-22.

\bibitem{Z3} Guangjun Zhu, Li Xu, Hong Wang and Zhongming Tang, Projective dimension and regularity of edge ideal of some weighted oriented graphs,  To appear in {\it Rocky MT J. Math.}.

\bibitem{Z4} Guangjun Zhu, Hong Wang, Li Xu and  Jiaqi Zhang,  Projective dimension and regularity of edge ideals of some vertex-weighted oriented m-partite graphs, {arXiv:1904.04682}.

\bibitem{Z5} Guangjun Zhu, Hong Wang, Li Xu and  Jiaqi Zhang, Projective dimension and regularity of  edge ideals of  some vertex-weighted oriented  unicyclic graphs, submitted.

\bibitem{Z6} Guangjun Zhu,  Li Xu,  Hong Wang,   and  Jiaqi Zhang, Projective dimension and  regularity of  powers of edge ideals of vertex-weighted  rooted forests, {arXiv:1904.03019}.

\bibitem{Z7} Guangjun Zhu, Hong Wang,   Li Xu and  Jiaqi Zhang, Regularity of  powers of edge ideals of vertex-weighted  oriented unicyclic graphs,  {\it arXiv:1904.02305}.


\end{thebibliography}
\end{document}